\documentclass[10pt,a4paper]{article}
\usepackage{amsfonts}
\usepackage[latin9]{inputenc}
\usepackage{amsmath}
\usepackage{amssymb}
\usepackage{breqn}
\usepackage{url}
\usepackage{graphicx}
\usepackage{verbatim}
\usepackage[pdfstartview=FitH]{hyperref}
\usepackage{amsthm}
\usepackage{hyperref}
\usepackage{url}
\usepackage{enumitem}
\usepackage{authblk}
\usepackage{mathtools}
\usepackage{appendix}
\usepackage{xcolor}

\newcommand{\hh}{\operatorname{\mathcal{H}}}
\newcommand{\im}{\operatorname{Im}}

\newcommand{\Aut}{\operatorname{Aut}}
\newcommand{\Span}{\operatorname{span}}

\newcommand {\ignore}[1]  {}

\newcommand{\type}{\operatorname{type}}

\makeatletter

\begin{document}
\newtheorem{theorem}{Theorem}[section]
\newtheorem{lemma}[theorem]{Lemma}
\newtheorem{definition}[theorem]{Definition}
\newtheorem{claim}[theorem]{Claim}
\newtheorem{example}[theorem]{Example}
\newtheorem{remark}[theorem]{Remark}
\newtheorem{proposition}[theorem]{Proposition}
\newtheorem{corollary}[theorem]{Corollary}
\newtheorem{observation}[theorem]{Observation}
\newtheorem{conjecture}[theorem]{Conjecture}
\newcommand{\subscript}[2]{$#1 _ #2$}
\newtheorem*{theorem*}{Theorem}
\author{Zohar Grinbaum-Reizis}
\affil{Department of Mathematics, Ben-Gurion University of the Negev, Be'er Sheva 84105, Israel, reiziszo@post.bgu.ac.il}
\author{Izhar Oppenheim}
\affil{Department of Mathematics, Ben-Gurion University of the Negev, Be'er Sheva 84105, Israel, izharo@bgu.ac.il}
\title{Curvature criterion for vanishing of group cohomology}
\maketitle

\begin{abstract}
We introduce a new geometric criterion for vanishing of cohomology for BN-pair groups. In particular, this new criterion yields a sharp vanishing of cohomology result for all BN-pair groups acting on non-thin affine building.
\end{abstract}

\section{Introduction}

In his seminal paper, Garland \cite{Garland} developed a machinery to prove vanishing of cohomology with real coefficients for groups acting on Bruhat-Tits buildings. This machinery, known today as ``Garland's method'', was generalized by Ballmann and \'{S}wi\c{a}tkowski \cite{Ballmann} to yield vanishing of cohomology with coefficients in unitary representations for groups acting properly and cocompactly on a simplicial complexes. These vanishing results had several applications (see \cite{Casselman}, \cite{Schneider} and more recently \cite{Andreas}) and Garland's method also had several applications in combinatorics (see \cite{Handbook}[Section 22.2] and references therein).  

The main idea behind Garland's method is that vanishing of cohomology can be deduced for a group acting on a simplicial complex, given that the spectral gaps of the links of the simplicial complex are large enough. Consequently, it implies vanishing of cohomology up to rank for a group acting on a Bruhat-Tits building if the thickness of the building is large enough. However, it was already noted in Garland's original paper \cite{Garland}, that his method does not yield a sharp result in the case of affine buildings. Namely, it was known before Garland (by the work of Kazhdan \cite{Kazhdan}) that a group with a proper, cocompact action on a non-thin affine building has property (T) (i.e., its first cohomology vanishes), but the thickness condition given in Garland's work did not hold for every non-thin affine building. In other words, there are cases of non-thin affine buildings that are not covered by Garland's criterion for vanishing of cohomology. Later, Casselman \cite{Casselman} was able to remove this restriction and prove vanishing of cohomology for every group acting on a non-thin affine buildings, but his proof used entirely different methods.

In \cite{Dymara}, Dymara and Januszkiewicz offered a different point of view of Garland's method: By assuming that fundamental domain is a single simplex, they showed that the spectral gap can be replaced with the notion of angle between subspaces. This change of perspective was very fruitful: Dymara and Januszkiewicz \cite{Dymara} used it to show vanishing of cohomology even when the stabilizers of vertices are not compact subgroups, Ershov, Jaikin and Kassabov pushed the idea of angle between subgroups to prove several results regarding property (T)(see \cite{Ershov}, \cite{Kassabov}, \cite{EJZK}) and the second-named author used these ideas to prove Banach versions of property (T) and vanishing of cohomology (see \cite{AveProjOpp}, \cite{VanBan}).

In this paper, we use the approach of Dymara and Januszkiewicz \cite{Dymara} and the ideas of Kassabov regarding angle between subspaces \cite{Kassabov} and get the following result for vanishing of cohomology for BN-pair groups:
\begin{theorem}
\label{vanishing coho for general building thm}
Let $G$ be a BN-pair group acting on a building $X$ such that $X$ is $n$-dimensional with $n \geq 2$ and all the $1$-dimensional links of $X$ are finite. Denote $C$ to be the cosine matrix of the Coxeter system associated with the Coxeter group that arises from the BN-pair of $G$ and $\widetilde{\mu}$ to be the smallest eigenvalue of $C$. If $X$ has thickness $\geq q+1$, where $q \geq 2$ and $\widetilde{\mu} > 1 - \frac{q+1}{2 \sqrt{q}}$, then:
\begin{enumerate}
\item For every continuous unitary representation $\pi$ of $G$, $H^k (X, \pi) =0$ for every $1 \leq k \leq n-1$.
\item If $1 \leq k \leq n-1$ is a constant such that all the $k$-dimensional links of $X$ are finite, then $H^i (G, \pi) = 0$ for every $1 \leq i \leq k$ and every continuous unitary representation $\pi$ of $G$. 
\end{enumerate}
\end{theorem}

As a corollary, we get another proof of the sharp vanishing result for groups acting on affine buildings of Casselman: 
\begin{corollary}
\label{Vanishing for affine buildings}
Let $G$ be a BN-pair group such that the building $X$ coming from the BN-pair of $G$ is an $n$-dimensional, non-thin affine building, with $n\geq 2$. Then for every unitary continuous representation $\pi$ of $G$, $H^k (G, \pi) =0$ for every $1 \leq k \leq n-1$.
\end{corollary}

This Theorem is a special case of a more general theorem that we will explain below after introducing the needed framework.

We start by introducing some terminology regarding simplicial complexes. Throughout, $X$ will denote a simplicial complex and $X(i)$ will denote the set of the $i$-dimensional simplices of $X$ (and we will use the convention that $X(-1) = \lbrace \emptyset \rbrace$). Below, we will use the following definitions: 
\begin{itemize}
\item The simplicial complex $X$ is called \textit{pure $n$-dimensional} if the top-dimensional simplices in $X$ are of dimension $n$ and every simplex in $X$ is contained in an $n$-dimensional simplex.
\item A pure $n$-dimensional simplicial complex $X$ is called \textit{gallery connected} if for every $\sigma, \sigma ' \in X(n)$, there is a sequence of $n$-dimensional simplices $\sigma = \sigma_1,..., \sigma_k = \sigma'$  such that for every $i$, $\sigma_i \cap \sigma_{i+1}$ is a simplex of dimension $n-1$. 
\item A pure $n$-dimensional simplicial complex $X$ is called $(n+1)$-\textit{partite} (or colorable) if there are disjoint sets of vertices $S_0,...,S_n$ of $X$ called \textit{the sides of $X$} such that every $\sigma \in X(n)$ has exactly one vertex in each side. 
\item For a $(n+1)$-partite simplicial complex $X$ with sides $S_0,...,S_n$, we define \textit{a type function}, denoted $\type : X \rightarrow 2^{\lbrace 0,...,n \rbrace}$, by $\type (\tau) = \lbrace i : \tau \cap S_i \neq \emptyset \rbrace$.
\item For a simplex $\sigma \in X$, the \textit{link of $\sigma$}, denoted $X_\sigma$ is the simplicial complex defined as $X_\sigma = \lbrace \tau \in X : \tau \cap \sigma = \emptyset, \tau \cup \sigma \in X \rbrace$ (by this definition, $X_\emptyset = X$). Note that if $X$ is pure $n$-dimensional and $(n+1)$-partite, then for every $\sigma \in X(i)$,  $X_\sigma$ is pure $(n-i-1)$-dimensional and $(n-i)$-partite.
\end{itemize}
 
Following Dymara and Januszkiewicz \cite{Dymara}, we work in the following setup: let $n \geq 2$ and $X$ be a pure $n$-dimensional, $(n+1)$-partite simplicial complex with sides $S_0,...,S_n$ and let $G$ be a closed subgroup of $\Aut(X)$ with respect to the compact-open topology. We consider the following properties for the couple $(X,G)$:
\begin{enumerate}[label=($\mathcal{B}${{\arabic*}})]
\item All the $1$-dimensional links are finite.
\item All the links of dimension $\geq 1$ are gallery connected. 
\item All the links are either finite or contractible (including $X$ itself). 
\item The group $G$ acts simplicially on $X$, such that the action is transitive on $X(n)$ and type preserving, i.e., for every $\tau \in X$ and every $g \in G$, $\type (\tau) = \type (g.\tau)$. 
\end{enumerate}

Next, we define the cosine matrix of $X$ that will play a central role in our criterion for vanishing of cohomology. In order to do so, we first recall some basic facts regarding simple random walks on graphs. For a finite graph $(V,E)$, the simple random walk on $(V,E)$ is an operator $M: \ell^2 (V) \rightarrow \ell^2 (V)$ defined as 
$$M \phi (v) = \frac{1}{d(v)} \sum_{u, \lbrace u,v \rbrace \in E} \phi (u),$$
where $d(v)$ is the valency of $v$, i.e., the number of neighbors of $v$. We further recall that $M$ is (similar to a) self-adjoint operator and as such has a spectral decomposition. Furthermore, all the eigenvalues of $M$ are in the interval $[-1,1]$ and if $(V,E)$ is connected, then $1$ is an eigenvalue of $M$ with multiplicity $1$ and all the other eigenvalues of $M$ are strictly smaller than $1$. Finally, we recall that if $(V,E)$ is a connected bipartite graph, then the second largest eigenvalue of $M$ is in the interval $[0,1)$. 

\begin{definition}[The cosine matrix of $X$]
\label{The cosine matrix of X def}
Let $n \geq 2$ and $X$ be a pure $n$-dimensional, $(n+1)$-partite simplicial complex with sides $S_0,...,S_n$ and let  $G$ be a closed subgroup of $\Aut(X)$ with respect to the compact-open topology. Assume that $(X,G)$ fulfill $(\mathcal{B} 1)- (\mathcal{B} 4)$ and define the cosine matrix of $X$, denoted $A = A (X)$ as follows: denote $\lambda_{i,j}$ to be the second largest eigenvalue of the random walk of $X_{\tau}$ for $\tau \in X(n-2), \type (\tau) = \lbrace 0,...,n \rbrace \setminus \lbrace i,j \rbrace$. Define $A$ to be the $(n+1) \times (n+1)$ matrix indexed by $\lbrace 0,...,n \rbrace$ as
$$A_{i,j} = \begin{cases}
1 & i =j \\
- \lambda_{i,j} & i \neq j
\end{cases}.$$ 
\end{definition}

\begin{remark}
We note that if $(X,G)$ fulfill $(\mathcal{B} 1)- (\mathcal{B} 4)$, then for every two simplices $\tau, \tau '$, if $\type (\tau) = \type (\tau ')$, then there is $g \in G$ such that $g. \tau = \tau '$ and in particular links $X_{\tau}$ and $X_{\tau '}$ are isomorphic as simplicial complexes and thus the matrix $A$ defined above is well-defined. 
\end{remark}

\begin{remark}
The reason we called $A$ the ``cosine matrix of $X$'' will be further explained below. For now we just note that the definition above coincides with the definition of the cosine matrix of a Coxeter group $G$ acting on a complex $X$ defined in \cite[Definition 6.8.11]{Davis} (see Definition \ref{Cosine coxter.M} below). 
\end{remark}  

Using the definition of the cosine matrix of $X$, we can state our main vanishing result:
\begin{theorem}
\label{General vanishing coho thm}
Let $n \geq 2$, $X$ be a pure $n$-dimensional, $(n+1)$-partite simplicial complex with sides $S_0,...,S_n$ and  $G$ be a closed subgroup of $\Aut(X)$ with respect to the compact-open topology. If $(X,G)$ fulfill $(\mathcal{B} 1)- (\mathcal{B} 4)$ and the cosine matrix of $X$ is positive definite, then:
\begin{enumerate}
\item For every continuous unitary representation $\pi$ of $G$, $H^k (X, \pi) =0$ for every $1 \leq k \leq n-1$.
\item If $1 \leq k \leq n-1$ is a constant such that all the $k$-dimensional links of $X$ are finite, then $H^i (G, \pi) = 0$ for every $1 \leq i \leq k$ and every continuous unitary representation $\pi$ of $G$. 
\end{enumerate}
\end{theorem}

In case where $X$ is a building, we will show in Section \ref{Vanishing of cohomology for groups acting on buildings sec} that the smallest eigenvalue of the cosine matrix of $X$ can be bounded from below by a function of the smallest eigenvalue of the Coxeter system and the thickness of the building. Using this fact, we will show that Theorem \ref{General vanishing coho thm} implies Theorem \ref{vanishing coho for general building thm} stated above.

\begin{remark}
The machinery developed in \cite{Dymara} allows also to compute the cohomology and not just prove vanishing. However, the statement of the computation includes introducing additional terminology and notation and therefore it is omitted from the introduction. A more general statements of Theorems \ref{vanishing coho for general building thm}, \ref{General vanishing coho thm} that include computation of cohomology (even when it does not vanish) appear in the body of this paper - see Theorems \ref{general statement buildings thm}, \ref{general statement of general thm}.
\end{remark}


While the Theorems stated above concern vanishing of group cohomology, the ideas of Dymara and Januszkiewicz \cite{Dymara} reduce this problem to showing a decomposition theorem in Hilbert spaces and the main tool that they used to prove such a decomposition was the idea of angle between subspaces. In the technical heart of this paper we use the results of Kassabov \cite{Kassabov} regarding angles between subspaces to prove a general decomposition theorem in Hilbert spaces that is interesting by its own right. After doing this, we show how this decomposition can be applied to deducing vanishing of cohomology in the general framework of Dymara and Januszkiewicz and how to apply this result for BN-pair buildings.  

\paragraph*{Geometric interpretation of Theorem \ref{General vanishing coho thm}.} Garland in his original paper used the term ``p-adic curvature'' for the eigenvalues of the random walks on the links. This term was used since the results were analogous to those of Matsushima who proved similar results for locally symmetric spaces. The condition for cohomology vanishing can be seen as a positive curvature condition as we will now explain. We start by recalling the following basic facts (see \cite[Chapters 6,7]{AVSBook}): let $v_0,...,v_n$ be points in general position in positive quadrant of the unit sphere of $\mathbb{R}^{n+1}$. These points can be thought as the vertices of a spherical simplex that is bounded by the subspaces $V_i = \Span \lbrace v_0,...,\hat{v_i},...,v_n \rbrace$. The cosine matrix of these subspaces, defined as:
$$A (V_0,...,V_n ) = \begin{cases} 
1 & i =j \\
- \cos (\angle (V_i, V_j)) & i \neq j 
\end{cases}.$$
In particular, the volume of the spherical simplex can be bounded from below by a function on the smallest eigenvalue of $A (V_0,...,V_n )$ (see \cite[Chapter 7, Proof of Theorem 2.1]{AVSBook}). 
Using the facts above as our geometric motivation, we note that the fact that cosine matrix of $X$, $A (X)$, is positive definite implies that there is a constant $\alpha >0$ such that for every unitary representation $(\pi, \mathcal{H})$ and every equivariant ``embedding'' of our simplicial complex $X$ in the unit sphere of $\mathcal{H}$, the spherical simplex spanned by the image of an $n$-simplex in $X$ has a spherical volume of at least $\alpha$. This statement is not precise, since our definition of an equivariant embedding is non-standard. A precise definition and an exact statement are given in the Appendix.

\paragraph*{Structure of this paper.} In Section \ref{Decomposition Theorem in Hilbert spaces sec}, we prove general a decomposition theorem in Hilbert spaces. In Section \ref{Vanishing of Cohomology for groups acting on simplicial complexes sec}, we show how this decomposition theorem implies vanishing of cohomology. In Section \ref{Vanishing of cohomology for groups acting on buildings sec}, we deduce a criterion for vanishing of cohomology for BN-pair groups and shows that in the case of affine buildings this criterion gives a sharp vanishing result. In the appendix, we give a further geometric interpretation for the vanishing criterion of Theorem \ref{General vanishing coho thm}.

\paragraph*{Acknowledgment.} The results of this paper are part of the M.Sc. project of the first-named author and she wishes to thank the Ben-Gurion University in which this research was conducted. The second-named author was partially supported by ISF grant no. 293/18.

\section{Decomposition Theorem in Hilbert spaces}
\label{Decomposition Theorem in Hilbert spaces sec}

Let $\mathcal{H}$ be a Hilbert space and let $V_0,...,V_n\subseteq \mathcal{H}$ be closed subspaces.

\begin{definition} 
\label{d1}
For a set $\tau \subseteq \lbrace 0,...,n \rbrace$ define  subspace $\mathcal{H}_\tau$ as
$$\mathcal{H}_\tau = \begin{cases}
\bigcap_{i \in \lbrace 0,...,n \rbrace \setminus \tau} V_i & \tau \neq \lbrace 0,...,n \rbrace \\
\mathcal{H} & \tau = \lbrace 0,...,n \rbrace
\end{cases},$$
\end{definition}
e.g., $\mathcal{H}_{\lbrace 0,...,n-1 \rbrace} = V_n$ and $\mathcal{H}_{\lbrace 0,...,n-2 \rbrace} = V_{n-1} \cap V_n$.
  
Note that for two sets $\tau, \eta \subseteq \lbrace 0,...,n \rbrace$, $\mathcal{H}_{\eta \cap \tau} = \mathcal{H}_\eta \cap \mathcal{H}_\tau$ and in particular if $\eta \subseteq \tau$, then $\mathcal{H}_\eta \subseteq \mathcal{H}_\tau$. Also note that $\mathcal{H}_{\emptyset} = \bigcap_{i=0}^n V_i$.
\begin{definition}
\label{d2}
For a set $\tau \subseteq \lbrace 0,...,n \rbrace$ define a subspace  $\mathcal{H}^\tau$ as 
$$\mathcal{H}^\tau = \begin{cases}
\mathcal{H}_{\tau} \cap \left( \bigcap_{\eta \varsubsetneqq \tau} \mathcal{H}_\eta^\perp \right) & \tau \neq \emptyset \\
\mathcal{H}_{\emptyset} & \tau = \emptyset 
\end{cases}$$
\end{definition}
and note that 
$$\bigcap_{\eta \varsubsetneqq \tau} \mathcal{H}_\eta^\perp = \left( \sum_{\eta \varsubsetneqq \tau} \mathcal{H}_\eta \right)^\perp.$$
\begin{definition}
[Angle between subspaces]\cite[Definition 3.2, Remark 3.19]{Kassabov} Let $V_1$ and $V_2$ be two closed subspaces in a Hilbert space. The cosine of $\angle  (V_1,V_2)$ is defined as $0$ if $V_1\subseteq V_2 \; or \; V_2\subseteq V_1$ and otherwise as
$$
\cos \angle (V_1,V_2)=
\sup \{\vert \left\langle v_1,v_2\right\rangle\vert \; :\;\Vert v_i\Vert =1, v_i\in V_i, v_i\perp (V_1\cap V_2)\}.$$
\end{definition}

\begin{remark}
There is an alternative definition of $\cos \angle (V_1,V_2)$ in the language of projections: denote $P_{V_1}, P_{V_2}, P_{V_1 \cap V_2}$ to be the orthogonal projections on $V_1,V_2, V_1 \cap V_2$. Then 
$$\cos \angle (V_1,V_2) = \Vert P_{V_1} P_{V_2} - P_{V_1 \cap V_2} \Vert.$$
The proof of the equivalence between this definition and the one given above is straightforward and can be found in \cite[Lemma 9.5]{Deutsch}. 
\end{remark}

\begin{definition}
Let $V_0,...,V_n$ be closed subspaces in a Hilbert space.
The cosine matrix $A=A(V_0,...,V_n)$ of $V_0,...,V_n$ is defined as follows: $A$ is $(n+1) \times (n+1)$ matrix with
$$A_{i,j}=\begin{cases} 1 & i=j \\
-\cos \angle (V_i,V_j) & i\ne j \end{cases}.$$
\end{definition}

\begin{theorem}[Decomposition Theorem]
\label{Decomposition Thm}
Let $\mathcal{H}$ be a Hilbert space,  $V_0,...,V_n\subseteq \mathcal{H}$ be closed subspaces, and  $A$ be the cosine matrix of $V_0,...,V_n$. If $A$ is positive definite, then for every $\tau \subseteq \{0,1,...,n\}$, it holds that $\mathcal{H}_{\tau }=\bigoplus_{\eta\subseteq \tau }\mathcal{H}^{\eta }$.
\end{theorem}

\begin{remark}
The original proof of this Theorem published in \cite{GRO} contained a mistake which is corrected in this version of the paper.
\end{remark}

The proof of this Theorem will require some set-up. We will start with defining an order relation between matrices that will be useful later on:
\begin{definition}
Let $A,B$ be two square matrices of the same dimensions. We denote $A \preceq B$ if for every $i,j$, $A_{i,j} \leq B_{i,j}$.
\end{definition} 

The reason to define this order relation is the following:
\begin{proposition}
\label{smallest e.v. proposition}
Let $A_1,A_2$ be two square matrices of the same dimension such that they both have $1$'s along the main diagonal and all their other entries are non-positive. Denote $\mu_i$ to be the smallest eigenvalues of $A_i$ for $i=1,2$. If $A_1 \preceq A_2$, then $\mu_1 \leq \mu_2$.
\end{proposition}

\begin{proof}
Note that $I-A_1, I-A_2$ are both non-negative matrices and as such, by the Perron-Frobenius Theorem, their largest eigenvalues are achieved by a vectors with non-negative entries. Thus, for $i=1,2$,
$$\mu_i = \max \lbrace \overline{v}^t A_i \overline{v} : \Vert \overline{v} \Vert =1, \overline{v} \text{ has non-negative entries} \rbrace.$$
The assumption $A_1 \preceq A_2$ implies that for every vector $\overline{v}$ with non-negative entries $\overline{v}^t A_1 \overline{v} \leq \overline{v}^t A_2 \overline{v}$ and therefore $\mu_1 \leq \mu_2$.
\end{proof}

The following result is proved in \cite{Kassabov}:
\begin{lemma}\cite[Lemma 4.2]{Kassabov}
\label{angle between intersections lemma}
The angles 
$\angle (V_1\cap V_3 , V_2\cap V_3)$
satisfy the inequality:
$\cos\angle(V_1\cap V_3 , V_2\cap V_3) \leq \frac{\lambda_{12}+\lambda_{13}\lambda_{23}}{\sqrt{1-\lambda^2_{13}}\sqrt{1-\lambda^2_{23}}},$ where $\lambda_{ij}=\cos \angle (V_i,V_j).$
\end{lemma}

This Lemma motivates the following definition:

\begin{definition}
Let $V_0,...,V_n$ be closed subspaces in a Hilbert space. Denote ${A'}$ be the $n\times n$ 
matrix defined as follows: 
$${A'}_{i,j} = \begin{cases}
1 & i =j \\
-\delta_{ij} & i \neq j
\end{cases}$$
where $\delta_{ij}=\frac{\lambda_{ij}+\lambda_{in}\lambda_{jn}}{\sqrt{1-\lambda_{in}^2}\sqrt{1-\lambda_{jn}^2}},$ and $\lambda_{ij}=\cos \angle (V_i,V_j)$.
\end{definition}

Additionally, we will need the following Lemma:
\begin{lemma}
\label{p.d. lemma}
For $V_0,...,V_n$, $A$, $A'$ be as above, let $\mu$ be the smallest eigenvalue of $A$ and $\mu'$ be the smallest eigenvalue of $A'$. If $A$ is positive definite, then $\mu \leq \mu'$. In particular, if $A$ is positive definite, then $A'$ is also positive definite. 
\end{lemma}
\begin{proof}
Let $A''$ be the $n\times n$ matrix defined by 
$$A''_{i,j}=\begin{cases} 1-\lambda^2_{in} & i=j \\
-\lambda_{ij}-\lambda_{in} \lambda_{jn} & i\ne j
\end{cases},$$ 
where $\lambda_{ij}=\cos \angle (V_i,V_j).$
As observed in the proof of \cite[Theorem 5.1 (a)]{Kassabov}, the matrices $A'$ and $A''$ has the following relation $A'=DA''D$
where $D$ is a diagonal matrix with entries $D_{i,i} = \frac{1}{\sqrt{1-\lambda^2_{in}}}.$ 
Thus, it is enough to prove that if $\alpha$ is the smallest positive eigenvalue of  $A''$ then $\mu \leq \alpha .$

Let
$\overline{u}=(u_0\; u_1\; .\; .\; .\; u_{n-1})^t ,\; \Vert \overline{u}\Vert =1,$
be an eigenvector with the eigenvalue $\alpha$. Denote $B=\begin{pmatrix}
Id&0\\-\lambda_n^t&1
\end{pmatrix}$
where $\lambda_n = (\lambda_{0 n} \; .\; .\; .\; \lambda_{n-1 n})^t$. Further denote
$\overline{v}=B^{-1}\begin{pmatrix}
\overline{u}\\0
\end{pmatrix}.$
By the definition of $\overline{v},$
$$\overline{v}=B^{-1}\begin{pmatrix}
\overline{u}\\0
\end{pmatrix}=\begin{pmatrix}
\overline{u}\\ \lambda_{0 n}u_0+..+\lambda_{n-1 n}u_{n-1}
\end{pmatrix}.$$
Thus,
\begin{equation}
\label{eq:1}
\Vert \overline{v}\Vert ^2=\Vert \overline{u}\Vert ^2+\vert \lambda_{0n}u_0+..+\lambda_{n-1 n}u_{n-1}\vert ^2\geq \Vert \overline{u}\Vert ^2=1 .
\end{equation}

As observed in the proof of  \cite[Theorem 5.1]{Kassabov}, $A$ can be written as the product:
$$A=\begin{pmatrix}
Id_{n\times n}&-\lambda_n\\0&1
\end{pmatrix}\begin{pmatrix}
A''&0\\0&1
\end{pmatrix}\begin{pmatrix}
Id_{n\times n}&0\\ -\lambda^t_n &1
\end{pmatrix}.$$ 
Hence,
\begin{align*}
\mu\Vert \overline{v}\Vert ^2 \leq \langle A\overline{v},\overline{v}\rangle =\left \langle B^t\begin{pmatrix}
A''&0\\0&1
\end{pmatrix}B\overline{v},\overline{v}\right \rangle=\left \langle \begin{pmatrix}
A''&0\\0&1
\end{pmatrix}B\overline{v},B\overline{v}\right \rangle\\=
\left\langle \begin{pmatrix}
A''&0\\0&1
\end{pmatrix}\begin{pmatrix}
\overline{u}\\0
\end{pmatrix},\begin{pmatrix}
\overline{u}\\0
\end{pmatrix}\right\rangle=\alpha \Vert \overline{u}\Vert ^2=\alpha.
\end{align*}
Together with
 \eqref{eq:1},
it follows that
$\mu \leq \frac{1}{\Vert \overline{v}\Vert ^2}\alpha \leq \alpha$ as needed.
\end{proof}

\begin{corollary}
\label{A (V_0 cap V_n,...) coro}
Let $V_0,...,V_n, \mathcal{H}$ be as above. If $A(V_0,...,V_n)$ is positive definite and the smallest eigenvalue of $A(V_0,...,V_n)$ is greater or equal to $\mu$, then $A(V_0 \cap V_n,...,V_{n-1} \cap V_n)$ is positive definite and the smallest eigenvalue of $A(V_0 \cap V_n,...,V_{n-1} \cap V_n)$ is greater or equal to $\mu$.
\end{corollary}

\begin{proof}
By Lemma \ref{angle between intersections lemma}, $A' \preceq A(V_0 \cap V_n,...,V_{n-1} \cap V_n)$ and the corollary follows from Lemma \ref{p.d. lemma} and Proposition \ref{smallest e.v. proposition} above.
\end{proof}

\begin{lemma}
\label{direct sum}
Let $\hh$ be a Hilbert space, $V_0,...,V_n\subseteq \hh$ be a closed subspaces, and $A$ be a cosine matrix of $V_0,...,V_n.$
Define 
\begin{align*}
&W_0=V_0, \\
 &W_1=V_1\cap (V_1\cap V_0)^{\perp },\\
 &.\\&.\\&.&\\ 
 &W_n=V_n\cap(\cap _{j<n}(V_n\cap Vj)^{\perp }) .
\end{align*}
If $A$ is positive definite, then $V_0+...+V_n=W_0\oplus...\oplus W_n$
\end{lemma}

\begin{proof}
Throughout this proof,  for a subspace $U\subseteq \hh$,  we denote $P_{U}$ to be the orthogonal projection on $U$. 

First,  we will prove that $V_0+...+V_n=W_0+...+W_n$ by induction on $n$.

For $n=0,$ it holds by the definition $W_0=V_0$.  Let $n > 0$ and assume that the equality holds for $n-1$.  We note that $W_0+...+W_n\subseteq V_0+...+V_n$ thus it remains to prove that for every $v\in V_0+...V_n$ it holds that $v\in W_0+...+W_n$.  Let $v\in V_0+...+V_n$, $v'\in V_0+...+V_{n-1}$ and $v''\in V_n$ such that $v=v'+v''$.  

Denote 
$$u'=v'+P_{\sum_{j<n}V_n\cap V_j}v'',$$ 
$$u''=P_{\left (\sum_{j<n}V_n\cap V_j\right )^{\perp }}v''$$ 
and note that $v=u'+u''$.  Observe that 
$$\sum_{j<n}V_n\cap V_j\subseteq V_0+...+V_{n-1}$$ 
and therefore $u'\in V_0+...+V_{n-1}=W_0+...+W_{n-1}$.  In order to finish this part of the proof,  we will show that $u'' \in W_n$.  Recall that 
$$\left (\sum_{j<n}V_n\cap V_j\right )^{\perp }=\bigcap _{j<n} \left (V_n \cap V_j \right )^{\perp }$$
and thus 
$$u'' \in \left(\bigcap _{j<n}(V_n\cap V_j)^{\perp }\right)$$
and in order to show that $u'' \in W_n$,  we are left to show that $u'' \in V_n$.  Note that 
$$\left (\sum_{j<n}V_n\cap V_j\right)\subseteq V_n$$ 
and $v''\in V_n$,  thus follows that
$$u''=P_{\left( \sum_{j<n}V_n\cap V_j\right)^{\perp }}v''=v''-P_{\sum_{j<n}V_n\cap V_j}v''\in V_n,$$
as needed.

Second, we will prove that $W_0+...+W_n=W_0\oplus...\oplus W_n$.
Let $w_0\in W_0,...,w_n\in W_n$ such that $w_0+...+w_n=0$.  For every $0\leq j<i\leq n$,  note that $w_i \in V_i$ and $v_j \in V_j \cap (V_i \cap V_j)^\perp$ and thus
$$\vert \langle w_j,w_i\rangle \vert \leq \cos \angle(V_j,V_i)\Vert w_j\Vert \Vert w_i\Vert .$$ 

Let $\mu$ be the smallest eigenvalue of $A$, then
\begin{align*}
&0=\Vert w_0+...+w_n\Vert ^2\geq \sum_{j=0}^n\Vert w_i\Vert ^2-2\sum_{0\leq j<i\leq n}\vert \langle w_j,w_i\rangle \vert\geq \\
&\sum_{j=0}^n\Vert w_i\Vert ^2-\sum_{0\leq j<i\leq n}2\cos \angle (V_j,V_i)\Vert w_j\Vert w_i \Vert =\\
&\begin{pmatrix}
\Vert w_0 \Vert &.&.&.&\Vert w_n\Vert
\end{pmatrix} A \begin{pmatrix}
\Vert w_0 \Vert &.&.&.&\Vert w_n\Vert
\end{pmatrix}^t\geq \mu \left (\sum _{j=0}^n\Vert w_j \Vert ^2\right ).
\end{align*}
Therefore,  for every $0\leq j\leq n$ it holds that $w_j=0$  as needed.
\end{proof}

\begin{corollary}
\label{part of d.sum}
Under the assumptions and the notation above, it holds that for every $j$ that
$$V_0+...+V_j=\left (V_0+...+V_{j-1}\right )\oplus W_j.$$
\end{corollary}

After this set-up,  we can prove the Decompositiom Theorem (Theorem \ref{Decomposition Thm} above):

\begin{proof} 
Let $\mathcal{H}_\tau, \mathcal{H}^\tau, \tau \subseteq \lbrace 0,...,n \rbrace$ be as above. 

It stems from the definition that for every $\tau \subseteq \{0,...,n\}$,  it holds that $\hh _{\tau }=\sum_{\eta \subseteq \tau }\hh ^{\eta }$ and we are left to prove this is a direct sum.  Also, without loss of generality, it is enough to prove that $\hh =\oplus _{\eta \subseteq \{0,...,n\}}\hh ^{\eta }.$ We will prove this decomposition by induction on $n$. 

For $n=0$, it holds by definition that $\hh ^{\emptyset }=\hh _{\emptyset }=V_0$ and $\hh _{\{0\}}=\hh$.  Thus, $\hh ^{\{0\}}=\hh \cap V_0^{\perp }=V_0^{\perp }$ and obviously $\hh =V_0\oplus V_0^{\perp }$.

Assume that $n>0$ and that the decomposition holds for $n-1$ subspaces.  Let $V_0,...,V_n$ be subspaces of a Hilbert space $\hh$ such that $A(V_0,...,V_n)$ is positive definite.  Fix $\{v_{\eta}\in \hh ^{\eta }\}_{\eta \subseteq \{0,...,n\}}$ such that $\sum_{\eta \subseteq \{0,...,n\}}v_{\eta }=0$.  We will show that $v_{\eta }=0$ for every $\eta$. 

First, by definition $v_{\{0,...,n\}}, v_{\emptyset }\perp v_{\eta }$ for every $\eta \subseteq \{0,...,n\}$ and therefore $v_{\{0,..,n\}}=v_{\emptyset}=0.$ Next, we will show that $v_{\eta }=0$ by downwards induction on $\vert \eta \vert$.

For $\vert \eta \vert =n$, without loss of generality,  it is enough to prove that $v_{\{0,...,n-1\}}=0$.  
Note that $v_{\{0,...,n-1\}}\in \hh^{\{0,...,n-1\}}=W_n$,  where $W_{i}$ is  defined as in  Lemma \ref{direct sum}, and 
$$\sum_{\substack{\emptyset \ne \eta \subseteq \{0,...,n\}\\ \eta \ne\{0,...,n\}, \{0,...,n-1\}}}v_{\eta }\in V_0+...+V_{n-1}.$$ 
By  Corollary \ref{part of d.sum}, 
$$V_0+...+V_n=\left (V_0+...+V_{n-1}\right )\oplus W_n$$ 
and thus, $v_{\{0,...,n-1\}}=0$.

Fix $1 \leq k \leq n-1$ and assume that for every $\vert \eta \vert \geq k+1,$ we have that $v_{\eta }=0.$ We will prove for every $\eta$ with $\vert \eta \vert=k$ it holds that $v_{\eta }=0$. Without loss of generality,  it is enough to prove that $v_{\{0,...,k-1\}}=0.$
Note that $v_{\{0,...,k-1\}}\in V_k,$ and let 
$$u=\sum_{\substack{\emptyset \ne \eta \subseteq \{0,...,n\}\\\vert \eta \vert \leq k, \eta \ne \{0,...,k-1\}}}v_{\eta}.$$

Observe that for every $\emptyset \ne \eta \subseteq \{0,...,n\}$ such that $\vert \eta \vert \leq k$ and $\eta \ne \{0,...,k-1\}$,  it holds that 
$$v_\eta \in V_0+...+V_{k-1}.$$
Thus $u \in V_0+...+V_{k-1}$ and also
\begin{align*}
&u+P_{\left(\sum_{j<k}V_j\cap V_k \right)} v_{\{0,...,k-1\}}\in V_0+...+V_{k-1}.
\end{align*}
Also observe that
 $$P_{\left (\sum_{j<k}V_j\cap V_k\right )^{\perp }}v_{\{0,...,k-1\}}\in W_k.$$  
 Therefore it holds that 
 $$\left(u + P_{\left(\sum_{j<k}V_j\cap V_k \right)} v_{\{0,...,k-1\}} \right) + _{\left (\sum_{j<k}V_j\cap V_k\right )^{\perp }}v_{\{0,...,k-1\}}$$
and by Corollary \ref{part of d.sum}, 
$$P_{\left (\sum_{j<k}V_j\cap V_k\right )^{\perp }}v_{\{0,...,k-1\}}=0.$$ 

Thus,
$$v_{\{0,...,k-1\}}-P_{\sum_{j<k}V_j\cap V_k}v_{\{0,...,k-1\}}=0$$ 
and 
$$v_{\{0,...,k-1\}}=P_{\sum_{j<k}V_j\cap V_k}v_{\{0,...,k-1\}}.$$
Define the following subspaces: 
$$V'_0=V_0\cap V_k,...,\widehat{V_k\cap V_k},...,V'_{n-1}= V_n\cap V_k$$ and 
$\hh '=V'_0+...+V'_{n-1}.$ 

By the induction assumption and Corollary \ref{A (V_0 cap V_n,...) coro},  
$$\hh'=\bigoplus_{\eta \subseteq \{0,...,n-1\}}\left (\hh '\right )^{\eta }.$$
Note that,
$$\im\left (P_{\sum_{j<k}V_j\cap V_k}\right )\subseteq V'_0+...+V'_{k-1}=\bigoplus_{\eta \subseteq \{0,...,n-1\},  \lbrace 0,..., k-1 \rbrace \nsubseteq \eta }\left (\hh '\right )^{\eta },$$
and,
\begin{align*}
v_{\{0,...,k-1\}}\in \left( \bigcap_{n \geq j \geq k} V_k \right) \cap \bigcap_{0 \leq i \leq k-1} \left( \left(\bigcap_{n \geq j \geq k} V_k \right) \cap V_i  \right)^{\perp} = \\
\left(\bigcap _{n \geq j>k}\left (V_j\cap V_k\right ) \right) \cap \bigcap_{0 \leq i \leq k-1} \left( \left(\bigcap_{n \geq j > k} V_j \cap V_k \right) \cap (V_i \cap V_k)  \right)^{\perp} = \\
\left(\bigcap _{n-1 \geq j>k-1}  V_j' \right) \cap \bigcap_{0 \leq i \leq k-1} \left( \left(\bigcap_{n-1 \geq j > k-1} V_j ' \right) \cap (V_i ')  \right)^{\perp}  = \left( \hh ' \right)^{\lbrace 0,...,k-1 \rbrace}. 
\end{align*}

Thus, $v_{\{0,...,k-1\}}$ and $P_{\sum_{j<k}V_j\cap V_k} v_{\{0,...,k-1\}}$ are in a different components of the decomposition of $\hh' $ as a direct sum and since they are equal,  it follows that they are both $0$.  Thus $v_{\{0,...,k-1\}}=0$ as needed.
\end{proof}

\section{Vanishing of Cohomology for groups acting on simplicial complexes}
\label{Vanishing of Cohomology for groups acting on simplicial complexes sec}

The aim of this section is to prove Theorem \ref{General vanishing coho thm} that appeared in the introduction and gave a criterion for cohomology vanishing for groups acting on simplicial complexes (under the assumptions $(\mathcal{B} 1)- (\mathcal{B} 4)$).

Let $n \geq 2$, $X$ be a pure $n$-dimensional, $(n+1)$-partite simplicial complex with sides $S_0,...,S_n$ and  $G$ be a closed subgroup of $\Aut(X)$ with respect to the compact-open topology. Assume that $(X,G)$ fulfill $(\mathcal{B} 1)- (\mathcal{B} 4)$ and fix $\triangle \in X(n)$. For a simplex $\tau \subseteq \triangle$, we denote $G_\tau$ to be the subgroup of $G$ stabilizing $\tau$ (and use the convention $G_\emptyset = G$). For $0 \leq i \leq n$, let $\sigma_i \subseteq \triangle$ be the $n-1$ dimensional face of $\triangle$ such that $\type (\sigma_i) = \lbrace 0,...,\hat{i},...,n \rbrace$. 

Given a continuous unitary representation $\pi$ of $G$ on a Hilbert space $\mathcal{H}$, define $V_i = V_i (\pi)$ as
$$V_i = \mathcal{H}^{\pi (G_{\sigma_i})} = \lbrace v \in \mathcal{H} : \forall g \in G_{\sigma_i}, \pi (g).v =v \rbrace.$$
Also for every $\tau \subseteq \triangle$ define
$$\mathcal{H}_\tau = \mathcal{H}_{\type (\tau)}, \mathcal{H}^\tau = \mathcal{H}^{\type (\tau)},$$
where $\mathcal{H}_{\type (\tau)}, \mathcal{H}^{\type (\tau)}$ are defined as in Definitions \ref{d1}, \ref{d2} above.

The following definitions of the core complex and the Davis Chamber appear in the paper of Dymara and Januszkiewicz \cite{Dymara}. The inspiration to in is attributed in  \cite{Dymara} to a similar construction of M.W. Davis in the setting of Coxeter complexes.
\begin{definition}\cite[Definition 1.3]{Dymara}
Let $X$ be a simplicial complex. Take the first barycentric subdivision $X'$ of $X.$ The core $X_D$ of $X$ is the subcomplex of $X',$ consisting of the simplices, spanned by barycenters of simplices of $X$ with compact links.
\end{definition}

\begin{definition}\cite[Definition 1.5]{Dymara}
Assume that $(X,G)$ fulfill $(\mathcal{B} 1)-(\mathcal{B} 4)$. Let $\triangle$ be a chamber of $X$ and let $\triangle '$ be the first barycentric subdivision of $\triangle$. The \textit{Davis chamber} $D$ is the subcomplex of $\triangle '$ consisting of simplices whose vertices are barycenters of simplices of $\triangle$ with finite links in $X$.

For any $\sigma \subset \Delta ,$ denote by $\Delta _{\sigma }$ the union of faces of $\Delta $ not containing $\sigma ,$ and put $D_{\sigma }=D\cap \Delta _{\sigma }.$
\end{definition}

Dymara and Januszkiewicz proved the following condition for vanishing of  $H^*(G,\pi )$:
\begin{theorem}\cite[Theorems 5.1, 5.2]{Dymara}
\label{DJ thm}
Assume that $(X,G)$ fulfill $(\mathcal{B} 1)-(\mathcal{B} 4)$ and let $\triangle$ be a chamber in $X$. Given a continuous unitary representation $\pi$ of $G$ on a Hilbert space $\mathcal{H}$, if for every $\tau \subsetneq \triangle$, it holds that $\mathcal{H}_{\tau }=\bigoplus_{\eta\subseteq \tau }\mathcal{H}^{\eta }$, then:
\begin{enumerate}
\item 
$H^i (X, \pi) =0$ for every $1 \leq i \leq n-1$.
\item
If $1 \leq k \leq n-1$ is a constant such that all the $k$-dimensional links of $X$ are finite, then $H^i (G, \pi) = 0$ for every $1 \leq i \leq k$.
\end{enumerate}
\end{theorem}

Moreover, Dymara and Januszkiewicz also generalized their result and gave a formula for computation of the group cohomology:

\begin{theorem}\cite[Theorem 7.1]{Dymara}
\label{DJ thm2}
Assume that $(X,G)$ fulfill $(\mathcal{B} 1)-(\mathcal{B} 4)$ and let $\triangle$ be a chamber in $X$. Given a continuous unitary representation $\pi$ of $G$ on a Hilbert space $\mathcal{H}$, if for every $\tau \subsetneq \triangle$, it holds that $\mathcal{H}_{\tau }=\bigoplus_{\eta\subseteq \tau }\mathcal{H}^{\eta }$, then
$$H^* (G, \pi) = \bigoplus _{\sigma \subseteq \Delta }\widetilde{H}^{* -1}(D_{\sigma };\mathcal{H}^{\sigma }),$$
where $D_\sigma$ are the subcomplexes of the Davis chamber defined above. Moreover, these cohomology spaces are Hausdorff.
\end{theorem}

\begin{remark}
The Theorems stated above are not exactly formulated as it appears in \cite{Dymara}. First, in \cite{Dymara}, the subspaces $\mathcal{H}_\tau$ are defined a little differently, namely, for $\tau \subseteq \triangle$, $\mathcal{H}_\tau =  \mathcal{H}^{\pi (G_{\tau})}$. The discrepancy between the definitions is resolved by \cite[Proposition 4.1]{Dymara} that states that for every $\tau \subsetneq \triangle$, $G_\tau$ is generated by $\lbrace G_\sigma : \sigma \subseteq \triangle, \sigma \in X(n-1), \tau \subseteq \sigma \rbrace$ and thus for every $\tau \subsetneq \triangle$, 
$$\mathcal{H}_\tau = \bigcap_{\sigma \subseteq \triangle, \sigma \in X(n-1), \tau \subseteq \sigma} \mathcal{H}_\sigma.$$
Thus, 
$$\mathcal{H}_{\type (\tau)} = \bigcap_{\sigma \subseteq \triangle, \sigma \in X(n-1), \type (\tau) \subseteq \type (\sigma)} \mathcal{H}_\sigma
= \bigcap_{i \in \lbrace 0,...,n \rbrace \setminus \tau} V_i,$$
as needed.
Second, in \cite{Dymara}, the condition for Theorem \ref{DJ thm} stated above is given in as a bound the eigenvalues of the Laplacian on the $1$-dimensional links, but this bound is only used to prove the existence of the decomposition $\mathcal{H}_{\tau }=\bigoplus_{\eta\subseteq \tau }\mathcal{H}^{\eta }$ and \cite[Theorems 5.1, 5.2, 7.1]{Dymara} follows from that decomposition.
\end{remark}

\ignore{
\begin{theorem}\cite[Theorem 5.1,5.2]{Dymara}
Let $X,G,\triangle$ as above. Given a continuous unitary representation $\pi$ of $G$ on a Hilbert space $\mathcal{H}$, if for every $\tau \subsetneq \triangle$, it holds that $\mathcal{H}_{\tau }=\bigoplus_{\eta\subseteq \tau }\mathcal{H}^{\eta }$, then:
\begin{enumerate}
\item $H^k (X, \pi) =0$ for every $1 \leq k \leq n-1$.
\item If $1 \leq k \leq n-1$ is a constant such that all the $k$-dimensional links of $X$ are finite, then $H^i (G, \pi) = 0$ for every $1 \leq i \leq k$.
\end{enumerate} 
\end{theorem}
}
Combining Theorems \ref{DJ thm}, \ref{DJ thm2} with Theorem \ref{Decomposition Thm} leads to the following:
\begin{theorem}
\label{Vanishing condition for pi}
Let $X,G,\triangle$ as above. Given a continuous unitary representation $\pi$ of $G$ on a Hilbert space $\mathcal{H}$, if $A(V_0 (\pi),...,V_n (\pi))$ is positive definite, then:
\begin{enumerate}
\item
$H^i (X, \pi) =0$ for every $1 \leq i \leq n-1$.
\item
If $1 \leq k \leq n-1$ is a constant such that all the $k$-dimensional links of $X$ are finite, then $H^i (G, \pi) = 0$ for every $1 \leq i \leq k$.
\item
$ H^* (G, \pi) = \bigoplus _{\sigma \subseteq \Delta }\widetilde{H}^{* -1}(D_{\sigma };\mathcal{H}^{\sigma })$ and these cohomology spaces are Hausdorff.
\end{enumerate} 
\end{theorem} 

\begin{proof}
By Theorem \ref{Decomposition Thm}, if  $A(V_0 (\pi),...,V_n (\pi))$ is positive definite, then for every $\tau \subsetneq \triangle$, it holds that $\mathcal{H}_{\tau }=\bigoplus_{\eta\subseteq \tau }\mathcal{H}^{\eta }$ and the assertions regarding vanishing of cohomology follow from Theorem \ref{DJ thm}.
\end{proof}

Thus, in order to prove vanishing of cohomology, it is enough to prove that for every unitary representation $\pi$, the matrix $A(V_0 (\pi),...,V_n (\pi))$ is positive definite. It was already observed in \cite{Dymara}, that the angles between $V_i (\pi)$'s can be bounded by the second largest eigenvalues of the simple random walk on the  $1$-dimensional links $X$: 
\begin{lemma}\cite[Lemma 4.6, step 1]{Dymara} (see also \cite[Theorem 1.7]{Oppenheim} and \cite[Corollary 4.20]{AveProjOpp})
\label{e.v. bounds cosine lemma}
Let $X,G,\triangle$ as above. For $0 \leq i,j \leq n$, where $i \neq j$, denote $\lambda_{i,j}$ to be the second largest eigenvalue on the simple random walk on $X_\tau$, where $\tau \in X(n-2)$ with $\type (\tau) = \lbrace 0,...,n \rbrace \setminus \lbrace i,j \rbrace$. Then for every unitary representation $\pi$ on a Hilbert space $\mathcal{H}$, if $V_i = V_i (\pi), V_j = V_j (\pi)$ are defined as above, then 
$\cos \angle (V_i, V_j) \leq \lambda_{i,j}$. 
\end{lemma}

\begin{corollary}
\label{A(X) geq A (V_0,...)}
If $A(X)$ is the cosine matrix of $X$ defined in Definition \ref{The cosine matrix of X def}, then for every unitary representation $\pi$, $A(V_0 (\pi),...,V_n (\pi)) \succeq A(X)$. In particular, if $A(X)$ is positive definite, then by Proposition \ref{smallest e.v. proposition}, for every unitary representation $\pi$, $A(V_0 (\pi),...,V_n (\pi))$ is positive definite.
\end{corollary}

Using Corollary \ref{A(X) geq A (V_0,...)} and Theorems \ref{Decomposition Thm}, \ref{DJ thm}, \ref{DJ thm2}, we can prove the following more general form of Theorem \ref{General vanishing coho thm}:
\begin{theorem}
\label{general statement of general thm}
Let $n \geq 2$, $X$ be a pure $n$-dimensional, $(n+1)$-partite simplicial complex with sides $S_0,...,S_n$ and $G$ be a closed subgroup of $\Aut(X)$ with respect to the compact-open topology. If $(X,G)$ fulfill $(\mathcal{B} 1)- (\mathcal{B} 4)$ and the cosine matrix of $X$ is positive definite, then:
\begin{enumerate}
\item For every continuous unitary representation $\pi$ of $G$, $H^k (X, \pi) =0$ for every $1 \leq k \leq n-1$.
\item If $1 \leq k \leq n-1$ is a constant such that all the $k$-dimensional links of $X$ are finite, then $H^i (G, \pi) = 0$ for every $1 \leq i \leq k$ and every continuous unitary representation $\pi$ of $G$. 
\item
$ H^* (G, \pi) = \bigoplus _{\sigma \subseteq \Delta }\widetilde{H}^{* -1}(D_{\sigma };\mathcal{H}^{\sigma })$ and these cohomology spaces are Hausdorff.
\end{enumerate}
\end{theorem}

\begin{proof}
Let $X,G,\triangle$ as above and let $\pi$ be some unitary representation of $G$ on a Hilbert space $\mathcal{H}$. Assume cosine matrix of $X$ defined in Definition \ref{The cosine matrix of X def} is positive definite. Thus, by Corollary \ref{A(X) geq A (V_0,...)}, the matrix $A(V_0 (\pi),...,V_n (\pi))$ is also positive definite and by Theorem \ref{Decomposition Thm}, for every $\tau \subsetneq \triangle$,
$\mathcal{H}_{\tau }=\bigoplus_{\eta\subseteq \tau }\mathcal{H}^{\eta }$. Thus the three assertions stated above follow directly by Theorem \ref{DJ thm}.
\end{proof}

\section{Vanishing of cohomology for groups acting on buildings}
\label{Vanishing of cohomology for groups acting on buildings sec}

The aim of this section is to prove Theorem \ref{vanishing coho for general building thm} that appeared in the introduction and show that it can be used to prove Theorem \ref{Vanishing for affine buildings} regarding vanishing of cohomologies for groups acting on affine buildings. We start by recalling some definitions regarding Coxeter systems:

\begin{definition}[Coxeter matrix, Coxeter system]
A Coxeter matrix $M=(m_{s,t})$ on a finite set $S$ is an $S\times S$ symmetric matrix with entries in $\mathbb{N} \cup \{\infty \}$ such that 
$$m_{s,t}=\begin{cases}
1&  s=t\\
\geq 2 & otherwise
\end{cases}$$
A Coxeter matrix $M$ defines a Coxeter system $(W,S)$, where $W= \langle S \vert \mathcal{R} \rangle$ is a group generated by $S$ with relations
$\mathcal{R}=\{(st)^{m_{st}} : s,t \in S\}$ ($m_{st} = \infty$ means that no relation of the form $(st)^m$ is imposed).
\end{definition}

\begin{remark}
A standard fact regarding Coxeter systems is that every Coxeter system acts by type preserving automorphisms on a partite, pure $(\vert S \vert -1)$-dimensional simplicial complex $\Sigma (W,S)$ called the Coxeter complex (see \cite[Chapter 3]{ABrownBook}), such that $(\Sigma (W,S), W)$ fulfill $(\mathcal{B} 2)-(\mathcal{B} 4)$ and if $m_{s,t} < \infty$ for every $s,t \in S$, then $(\Sigma (W,S), W)$ also fulfill $(\mathcal{B} 1)$.
\end{remark}

\begin{definition}
\label{Cosine coxter.M}
\cite[definition 6.8.11]{Davis}
The cosine matrix associated to a Coxeter matrix $M$ is the $S\times S$ matrix $C=(c_{ij})$ defined by $c_{ij}=-\cos(\frac{\pi}{m_{s_i,s_j}}).$
When $m_{ij}=\infty$ we define $c_{ij}=-1.$
\end{definition}

\begin{observation}
Assume that $(W,S)$ is a Coxeter system such that $m_{s,t} < \infty$ for every $s,t \in S$. Denote $S = \lbrace s_0,...,s_n \rbrace$ and abbreviate $m_{s_i,s_j} = m_{i,j}$. 
Then $C$ defined above is exactly the cosine matrix of $\Sigma (W,S)$ defined in Definition \ref{The cosine matrix of X def}. Indeed, by \cite[Corollary 3.20]{ABrownBook}, for every $0 \leq i, j \leq n$, $i \neq j$, the link of type $\lbrace i,j \rbrace$ is a $2 m_{i,j}$-gon and it is easy to verify that the second largest eigenvalue of the simple random walk on a $2 m_{i,j}$-gon is $\cos(\frac{\pi}{m_{i,j}})$ and thus for every $i,j$, $c_{i,j}=- \cos(\frac{\pi}{m_{i,j}})$ (for $i=j$, $m_{i,i} =1$ and therefore $c_{i,i} = 1$). 
\end{observation}

\begin{lemma}
\label{bound on A in building lemma}
Let $X$ be a building of dimension $\geq 2$ and thickness $\geq q+1$ where $q \geq 2$ such that all the $1$-dimensional link of $X$ are compact (i.e., finite) and let $C$ be the cosine matrix of the Coxeter complex (i.e., the apartment) of the building $X$. Then for $A = A(X)$ the cosine matrix of $X$ it holds that
$$A \succeq 2\frac{\sqrt{q}}{q+1}C +(1-2\frac{\sqrt{q}}{q+1})I.$$
In particular, if we denote $\widetilde{\mu}$ to be the smallest eigenvalue of $C$, then $\widetilde{\mu} > 1 - \frac{q+1}{2 \sqrt{q}}$ implies that $A$ is positive definite.
\end{lemma}

\begin{proof}
Let $X$ be as above. As a simplicial complex, $X$ is a partite and we fix a type function on the vertices. For every $0 \leq i,j \leq n$, $i \neq j$, we denote $X_{i,j}$ to be the $1$-dimensional link $X_\tau$, where $\tau \in X(n-2), \type (\tau) = \lbrace 0,...,n \rbrace \setminus \lbrace i,j \rbrace$. Also denote $\lambda_{i,j}$ to be the second largest eigenvalue of the simple random walk on $X_{i,j}$. By our assumption $X_{i,j}$ is a finite graph and since $X$ is a building of thickness $\geq q+1$, $X_{i,j}$ is a $1$-dimensional spherical building of minimal degree $\geq q+1$. Let $M$ be the Coxeter matrix of the Coxeter system associated with $X$ such that $M$ is indexed according to our type function, i.e., the entries of $M$ are indexed by $\lbrace 0,...,n\rbrace$ and for every $i,j, i \neq j$, $m_{i,j}$ is the diameter of $X_{i,j}$. 

We recall that by a classical result of Feit and Higman \cite{FeitH}, for every $i \neq j$, if $X_{i,j}$ has a minimal degree $>2$, then $m_{i,j} \in \lbrace 2,3,4,6,8 \rbrace$ and in \cite[Theorem 7.10]{Garland} $\lambda_{i,j}$ was computed for every such choice of $m_{i,j}$. The computations of \cite[Theorem 7.10]{Garland} allow $X_{i,j}$ to be a bi-regular graph and $\lambda_{i,j}$ is computed according to the degrees of $X_{i,j}$, but here we will only state the results assuming that the degrees are greater or equal to $q+1$: Let $i,j, i \neq j$, 
\begin{itemize}
\item If $m_{i,j} = 2$, then $\lambda_{i,j} = 0$.
\item If $m_{i,j} = 3$, then $\lambda_{i,j} \leq \frac{\sqrt{q}}{q+1}$.
\item If $m_{i,j} = 4$, then $\lambda_{i,j} \leq \sqrt{2}\frac{\sqrt{q}}{q+1}$.
\item If $m_{i,j} = 6$, then $\lambda_{i,j} \leq \sqrt{3}\frac{\sqrt{q}}{q+1}$.
\item If $m_{i,j} = 8$, then $\lambda_{i,j} \leq \sqrt{2+\sqrt{2}}\frac{\sqrt{q}}{q+1}$.
\end{itemize}
We observe that in all the inequalities above, $\lambda_{i,j} \leq \cos(\frac{\pi}{m_{i,j}}) \frac{2\sqrt{q}}{q+1}$ and thus
$$A \succeq 2\frac{\sqrt{q}}{q+1}C +(1-2\frac{\sqrt{q}}{q+1})I,$$
as needed. 
\end{proof}

This Lemma readily applies the following more general form of Theorem \ref{vanishing coho for general building thm} that appeared in the introduction:
\begin{theorem}
\label{general statement buildings thm}
Let $G$ be a BN-pair group acting on a building $X$ such that $X$ is $n$-dimensional with $n \geq 2$ and all the $1$-dimensional links of $X$ are finite. Denote $C$ to be the cosine matrix of the Coxeter system associated with the Coxeter group that arises from the BN-pair of $G$ and $\widetilde{\mu}$ to be the smallest eigenvalue of $C$. If $X$ has thickness $\geq q+1$, where $q \geq 2$ and $\widetilde{\mu} > 1 - \frac{q+1}{2 \sqrt{q}}$, then:
\begin{enumerate}
\item For every continuous unitary representation $\pi$ of $G$, $H^k (X, \pi) =0$ for every $1 \leq k \leq n-1$.
\item If $1 \leq k \leq n-1$ is a constant such that all the $k$-dimensional links of $X$ are finite, then $H^i (G, \pi) = 0$ for every $1 \leq i \leq k$ and every continuous unitary representation $\pi$ of $G$. 
\item
$ H^* (G, \pi) = \bigoplus _{\sigma \subseteq \Delta }\widetilde{H}^{* -1}(D_{\sigma };\mathcal{H}^{\sigma })$ and these cohomology spaces are Hausdorff.
\end{enumerate}
\end{theorem}

\begin{proof}[Proof of Theorem \ref{vanishing coho for general building thm}]
Let $X, G$ be as in Theorem \ref{vanishing coho for general building thm} and let $C$ be the cosine matrix of the Coxeter complex (i.e., the apartment) of $X$. Denote by $\mu$ the smallest eigenvalue of $A(X)$ and by $\widetilde{\mu}$ the smallest eigenvalue of $C$. By our assumption, $X$ has thickness $\geq q+1$ and thus by Lemma \ref{bound on A in building lemma} and Proposition \ref{smallest e.v. proposition}, 
$$\mu \geq 2 \frac{\sqrt{q}}{q+1} \widetilde{\mu} + 1-2\frac{\sqrt{q}}{q+1}.$$
It follows that if $\widetilde{\mu} > 1 - \frac{q+1}{2 \sqrt{q}}$, then $\mu >0$, i.e., $A(X)$ is positive definite and the assertions above follow from Theorem \ref{general statement of general thm}.
\end{proof}

The sharp vanishing result for affine buildings stated in of Theorem \ref{Vanishing for affine buildings} is a consequence of Theorem \ref{vanishing coho for general building thm} and of a well-establish fact regrading the cosine matrix of affine Coxeter complexes:
\begin{proof}[Theorem \ref{Vanishing for affine buildings}]
Let $G$ be a BN-pair group such that the building $X$ coming from the BN-pair of $G$ is an $n$-dimensional, non-thin affine building, with $n\geq 2$ and let $C$ be the cosine matrix of the Coxeter complex (i.e., the apartment) of $X$. Since $X$ is non-thin it has thickness $\geq 3$ and therefore by Theorem \ref{vanishing coho for general building thm}, it is enough to prove that 
$\widetilde{\mu} > 1 - \frac{3}{2 \sqrt{2}}$, where $\widetilde{\mu}$ is the smallest eigenvalue of $C$. By \cite[Theorem 6.8.12]{Davis}, the cosine matrix of an affine Coxeter complex is positive semidefinite (with co-rank $1$) and thus 
$\widetilde{\mu} =0 > 1 - \frac{3}{2 \sqrt{2}}$ as needed.

Note that all the links of $X$ that are not $X$ itself are compact and therefore by Theorem \ref{General vanishing coho thm}, $H^i (G, \pi) =0$ for every $1 \leq i \leq n-1$ and every continuous unitary representation $\pi$.
\end{proof}

Another consequence of Theorem \ref{vanishing coho for general building thm} is vanishing of cohomology for Kac-Moody groups acting on Kac-Moody buildings given that the thickness is large enough and all the $1-$dimensional link are finite. This was already proved by Dymara and Januszkiewicz \cite{Dymara} where the condition on the thickness was that is should be greater or equal to $\frac{1}{25}(1764)^n$, where $n$ is the dimension of the building. Theorem \ref{vanishing coho for general building thm} shows that the same vanishing result holds under a much weaker condition on the thickness. In order to illustrate this point, we preform an exact calculation of a specific example of a (hyperbolic) Kac-Moody group (this example was chosen rather arbitrarily and one can preform similar computations for any specific example of a BN-pair group):

\begin{example}
Consider the Coxeter group whose Coxeter diagram is a square such that one of the edges is labelled $4$ and  remaining ones are labelled $3$. By the work of Tits \cite{Tits}, for every finite field $\mathbb{F}_q$, there is a BN-pair group $G (q)$ acting on a building $X$ with thickness $q+1$ and the Coxeter group as above. Note that according to the Coxeter diagram the links of the vertices are spherical building (see \cite[Section 4]{Dymara2}) and in particular the links of all the vertices are finite and we can apply Theorem \ref{vanishing coho for general building thm}. The cosine matrix of this Coxeter group is 
$$\begin{pmatrix}
1&-\frac{1}{\sqrt{2}}&-\frac{1}{2}&0\\
-\frac{1}{\sqrt{2}}&1&-\frac{1}{2}&0\\
-\frac{1}{2}&-\frac{1}{2}&1&-\frac{1}{2}\\
0&0&-\frac{1}{2}&1
\end{pmatrix}$$ 
and its smallest eigenvalue is $\frac{-\sqrt{2}+1}{2}$. Note that for every $q \geq 4$, it holds that $\frac{-\sqrt{2}+1}{2}>1-\frac{q+1}{2\sqrt{q}}$ and the conditions of Theorem \ref{vanishing coho for general building thm} are satisfied. Thus, in this example, for every $q \geq 4$ and every unitary representation of $G (q)$, it holds that $H^1 (G (q), \pi) = H^2 (G(q), \pi) =0$.  
\end{example}
\appendix

\section{Equivariant embedding interpretation of our criterion}
\label{Equivariant embedding interpretation of our criterion appnx}

The aim of this appendix is to give a geometric interpretation to our vanishing criterion, i.e., to give a geometric meaning to the condition that the cosine matrix of a complex $X$ is positive definite. Let $X$ be an $n$-dimensional simplicial complex and $G$ be a group acting simplicially on $X$ such that $(X,G)$ fulfill $(\mathcal{B} 1)- (\mathcal{B} 4)$. 

Let $(\pi, \mathcal{H})$ be a unitary representation of $G$. Below, we fix $\triangle = \lbrace x_0,...,x_n \rbrace \in X(n)$ use the notations of Section \ref{Vanishing of Cohomology for groups acting on simplicial complexes sec}, i.e.:
\begin{enumerate}
\item For $\tau \subseteq \triangle$, $G_\tau$ is the stabilizer subgroup of $\tau$ in $G$.
\item For $i \in \lbrace 0,...,n \rbrace$, $V_i (\pi) = \mathcal{H}^{\pi (G_{\lbrace x_k : k \in \lbrace 0,...,n \rbrace \setminus \lbrace i \rbrace \rbrace})}$.
\item For $\tau \subseteq \triangle$, $\mathcal{H}_\tau = \bigcap_{i \notin \tau} V_i (\pi) = \mathcal{H}^{\pi (G_\tau)}$. 
\end{enumerate}

\begin{lemma}
\label{generation lemma}
Let $(X,G)$, $ \triangle = \lbrace x_0,...,x_n \rbrace \in X(n)$ and $\lbrace G_\tau \rbrace_{\tau \subseteq  \triangle}$ be as above. For every $\tau, \tau ' \subseteq  \triangle$, $\langle G_\tau, G_{\tau '} \rangle = G_{\tau \cap \tau '}$.
\end{lemma}

\begin{proof}
If $\tau \subseteq \tau'$ or $\tau ' \subseteq \tau$, there is nothing to prove, thus we will assume that this is not the case and, in particular, that $\vert \tau \cap \tau ' \vert \leq n-1$. 

Observe that for every $\vert \tau \cap \tau ' \vert \leq n-1$, the couple $(X_{\lbrace x_j : j \in \tau \cap \tau ' \rbrace}, G_{\tau \cap \tau '} )$ also fulfill $(\mathcal{B} 1)- (\mathcal{B} 4)$ and if we fix $\lbrace x_j : j \in \lbrace 0,...,n \rbrace \setminus \tau \cap \tau ' \rbrace$ to be a fundamental domain of $X_{\lbrace x_j : j \in \tau \cap \tau ' \rbrace}$, we will get exactly the subgroups $\lbrace G_{\tau \cap \tau ' \cup \eta} \rbrace_{\eta \subseteq  \triangle \setminus \tau \cap \tau '}$. Thus, it is enough to prove that if $\tau \cap \tau ' = \emptyset$, it follows that $\langle G_\tau, G_{\tau '} \rangle = G$.

By definition, $G_\tau, G_{\tau ' } \subseteq G$, thus $\langle G_\tau, G_{\tau '} \rangle \subseteq G$. Next, we will show that for every $g \in G$, $g$ can be written as a product of elements in $G_\tau \cup G_{\tau '}$.

Fix some $g \in G$. By $(\mathcal{B} 2)$, $\triangle$ and $g.  \triangle$ are connected by a gallery, i.e., there are $n$-dimensional simplices $\sigma_1,...,\sigma_k$ such that $\sigma_j \cap \sigma_{j+1} \in X(n-1)$ and $ \triangle = \sigma_1,...,\sigma_k = g.  \triangle$. We will prove that if $k$ is the length of the connecting gallery, then $g \in  (G_\tau \cup G_{\tau '})^{k+1}$. 

The proof is by induction on $k$. If $k=0$, then $g. \triangle =   \triangle$ and thus $g \in G_{\triangle} \subseteq G_\tau$. 

Assume that our claim is true for $k-1$ and let $ \triangle = \sigma_1,...,\sigma_k = g.  \triangle$. Note that since $\sigma_1 \cap \sigma_2 \in X(n-1)$, there is some $i_0$ such that $\sigma_1 \cap \sigma_2 = \lbrace x_i : i \in  \triangle \setminus \lbrace i_0 \rbrace \rbrace$. By $(\mathcal{B} 4)$, it follows that there is some $g' \in G_{\triangle \setminus \lbrace i_0 \rbrace}$ such that $\sigma_1 = g'.\sigma_2$. Note that $G_{\triangle \setminus \lbrace i_0 \rbrace} \subseteq G_\tau \cup G_{\tau '}$ and thus $g ' \in G_\tau \cup G_{\tau '}$. Thus, $\sigma_2 = g'. \sigma_1,...,\sigma_k = g.\triangle$ is a gallery of length $k-1$ and also $\sigma_1, (g')^{-1}.\sigma_3,...,(g')^{-1}. \sigma_k = (g')^{-1}g. \triangle$ is a gallery of length $k-1$. It follows that $(g')^{-1}g \in (G_\tau \cup G_{\tau '})^{k}$ and therefore $g \in g' (G_\tau \cup G_{\tau '})^{k} \subseteq (G_\tau \cup G_{\tau '})^{k+1}$ as needed. 
\end{proof}
   
Let $\phi : X (0) \rightarrow \mathcal{H}$. We recall that the map $\phi$ is called  \textit{equivariant} (with respect to $\pi$) if for every $g \in G$ and every vertex $x$ of $X$ it holds that $\phi (g.x) = \pi (g) \phi (x)$.

\begin{observation}
Fix $\triangle = \lbrace x_0,...,x_n \rbrace \in X(n)$ and let $(\pi, \mathcal{H})$ be a unitary representation of $G$. We observe that by $(\mathcal{B} 4)$, an equivariant map $\phi : X (0) \rightarrow \mathcal{H}$ is uniquely determined by the choices of $\phi (x_i), i=0,...,n$ and that $\phi (x_i) \in \mathcal{H}^{\pi (G_{\lbrace x_i \rbrace})}$ for every $i=0,...,n$. Vice-versa, every choice $v_i \in \mathcal{H}^{\pi (G_{\lbrace x_i \rbrace})}, i=0,...,n$ defines an equivariant map $\phi$ via $\phi (g. x_i) = \pi (g)v_i$ (note that this indeed defines a well-defined equivariant map $\phi : X (0) \rightarrow \mathcal{H}$).
\end{observation}

\begin{proposition}
\label{general position prop}
Let $(\pi, \mathcal{H})$ be a unitary representation of $G$ without any non-trivial invariant vectors and $\phi : X(0) \rightarrow \mathcal{H}$ be an equivariant map such that $\phi (x) \neq 0$ for every $x \in X(0)$. Then for every $\tau \subsetneq \triangle$ and every $x_i \notin \tau$, $P_{\mathcal{H}_\tau} \phi (x_i) \neq \phi (x_i)$.
\end{proposition}

\begin{proof}
Assume toward contradiction that $P_{\mathcal{H}_\tau} \phi (x_i) = \phi (x_i)$, then by the observation $\phi (x_i) \in \mathcal{H}_\tau \cap \mathcal{H}_{\lbrace x_i \rbrace}$ and thus it is stabilized by both $G_\tau$ and $G_{\lbrace x_i \rbrace}$. By our assumption $x_i \notin \tau$ and thus $\phi (x_i)$ is stabilized by $\langle G_\tau, G_{\lbrace x_i \rbrace} \rangle = G_\emptyset = G$ (here we use Lemma \ref{generation lemma}). Therefore $\phi (x_i)$ is a non-trivial invariant vector and this contradicts our assumption. 
\end{proof}

We define $\phi : X (0) \rightarrow \mathcal{H}$ to be an \textit{equivariant embedding of $X$ into the unit sphere of $\mathcal{H}$} if the following holds:
\begin{enumerate}
\item For every $x \in X(0)$, $\Vert \phi (x) \Vert =1$.
\item The map $\phi$ is equivariant.
\item For every $i \neq j$, $P_{\mathcal{H}_{\triangle_{i,j}}} \phi (x_i) \in \Span \lbrace \phi (x) : x \in \triangle_{i,j} \rbrace$, where $\triangle_{i,j} = \lbrace x_k : k \in \lbrace 0,...,n \rbrace \setminus \lbrace i,j \rbrace \rbrace$. 
\end{enumerate}
The last condition may be thought of a being in general position with respect to the subspaces $V_i (\pi)$. 

The following Theorem gives a geometric interpretation of our criterion that the cosine matrix $A = A(X)$ is positive definite. Basically it states that for any representation $(\pi, \mathcal{H})$ with no non-trivial invariant vectors, an $n$-simplex of $X$ is mapped to a spherical simplex of that cannot be ``too small'' (there is a lower bound on the spherical $n$-volume of the image). 

\begin{theorem}
Assume that the cosine matrix $A=A(X)$ of $X$ is positive definite. Then there is a constant $\alpha >0$ that depends on the smallest positive eigenvalue of $A$ such that for every unitary representation $(\pi, \mathcal{H})$ without non-trivial invariant vectors and any $\phi : X(0) \rightarrow \mathcal{H}$ equivariant embedding of $X$ into the unit sphere of $\mathcal{H}$, $\lbrace \phi (x_0),...,\phi (x_n) \rbrace$ are vertices of an $n$-dimensional spherical simplex with a spherical $n$-volume of at least $\alpha$.
\end{theorem}

\begin{proof}
Let $(\pi, \mathcal{H})$ be some unitary representation and $\phi : X(0) \rightarrow \mathcal{H}$ equivariant embedding of $X$ into the unit sphere of $\mathcal{H}$.

By Proposition \ref{general position prop},  $\phi (x_0),...,\phi (x_n)$ are in general position in $\mathcal{H}$ and thus span an $n$-dimensional spherical simplex in the $(n+1)$-dimensional subspace $V' = \Span \lbrace \phi (x_k) : k \in \lbrace 0,...,n \rbrace \rbrace$. Restricting our attention to $V'$, the spherical simplex at hand is bounded by the subspaces $V_i ' = \Span \lbrace \phi (x_k) : k \in \lbrace 0,...,n \rbrace \setminus \lbrace i \rbrace \rbrace$. The cosine matrix of these subspaces is defined as above:
$$A (V_0 ',...,V_n ') = \begin{cases} 
1 & i =j \\
- \cos (\angle (V_i ', V_j ')) & i \neq j 
\end{cases}.$$
In particular, the volume of the spherical simplex can be bounded from below by a positive increasing function on the smallest eigenvalue of $A (V_0 ',...,V_n ')$ (see \cite[Chapter 7, Proof of Theorem 2.1]{AVSBook}). Therefore in order to prove the Theorem it is sufficient to show that the smallest eigenvalue of $A(X)$ is a lower bound for the smallest eigenvalue of $A (V_0 ',...,V_n ')$.

We note that for every $i,j$, $V_i ' \cap V_j ' = \Span \lbrace \phi (x_k) : k \in \lbrace 0,...,n \rbrace \setminus \lbrace i,j \rbrace \rbrace$ and $V_i ' \cap (V_i ' \cap V_j ')^\perp$ is the one dimensional space that can be written as
$$V_i ' \cap (V_i ' \cap V_j ')^\perp = \Span \lbrace \phi (x_i) - P_{V_i ' \cap V_j '} \phi (x_i) \rbrace.$$
By our definition of an equivariant embedding, $P_{V_i ' \cap V_j '} \phi (x_i) = P_{V_i (\pi) \cap V_j (\pi)} \phi (x_i)$ and thus for every $i,j$, $V_i ' \cap (V_i ' \cap V_j ')^\perp \subseteq V_i (\pi) \cap (V_i (\pi) \cap V_j (\pi))^\perp$. It follows that 
$$\cos (\angle (V_i ', V_j ')) \leq \cos (\angle (V_i (\pi), V_j (\pi))).$$
By Corollary \ref{A(X) geq A (V_0,...)}, $A (V_0 ',..., V_n ') \succeq A(X)$ and thus by Proposition \ref{smallest e.v. proposition} the smallest eigenvalue of $A (V_0',...,V_n ')$ is bounded from below by the smallest eigenvalue of $A (X)$ as needed.
\end{proof}

\bibliographystyle{plain}
\bibliography{bibl}
\end{document}